\documentclass[reqno, 11pt]{amsart}
\usepackage[utf8]{inputenc}
\usepackage{amsmath}
\usepackage{amsthm}
\usepackage{amssymb}
\usepackage{mathtools}
\usepackage{booktabs}
\usepackage{amsxtra}
\usepackage{amsbsy}
\usepackage{comment}
\usepackage{mathrsfs}
\usepackage{longtable}

\usepackage[english]{babel}
\usepackage{blindtext}
\usepackage{url}
\usepackage{bbm}
\usepackage{euscript}
\usepackage{pifont}
\usepackage{hyperref}
\usepackage{wasysym}
\hypersetup{
    colorlinks=true,
    linkcolor=blue,
    filecolor=magenta,
    urlcolor=cyan,
}
\usepackage{graphicx}
\usepackage{epstopdf}

    \usepackage[
    top    = 2.50cm,
    bottom = 2.50cm,
    left   = 2.81 cm,
    right  = 2.81 cm]{geometry}

\newtheorem{thm}{Theorem}[section]
\newtheorem{lem}[thm]{Lemma}
\newtheorem{cor}[thm]{Corollary}

\newtheorem{fact}[thm]{Fact}
\newtheorem*{conjecture*}{Conjecture}
\newtheorem*{thm*}{Theorem}

\theoremstyle{remark}
\newtheorem*{remark}{Remark}

\theoremstyle{definition}
\newtheorem{define}[thm]{Definition}

\newcommand{\RNum}[1]{\uppercase\expandafter{\romannumeral #1\relax}}
\newcommand{\Al}{\textup{\textsf{A}}}
\newcommand{\Sy}{\textup{\textsf{S}}}
\newcommand{\Sym}{{\operatorname{Sym}}}
\newcommand{\Stab}{{\operatorname{Stab}}}

\newskip\aline \newskip\halfaline
\title[Set Orbits of Permutation Groups]{Finite Permutation Groups with Few Orbits Under the Action on the Power Set}
\author[A. Betz]{Alexander Betz}
\author[M. Chao-Haft]{Max Chao-Haft}
\author[T. Gong]{Ting Gong}
\author[T. Keller]{Thomas Michael Keller}
\author[A. Ter-Saakov]{Anthony Ter-Saakov}
\author[Y. Yang]{Yong Yang}
\address{Department of Mathematics, Le Moyne College, 1419 Salt Springs Road, Syracuse, NY 13214}
\email{betzas@lemoyne.edu}
\address{Department of Mathematics, Harvey Mudd College, 340 East Foothill Boulevard, Claremont, CA 91711}
\email{mchaohaft@g.hmc.edu}
\address{Department of Mathematics, University of Notre Dame, 255 Hurley, Notre Dame, IN 46556}
\email{tgong@nd.edu}
\address{Department of Mathematics, Texas State University, 601 University Drive, San Marcos, TX 78666}
\email{keller@txstate.edu}
\address{Department of Mathematics $\&$ Statistics, Boston University, 111 Cummington Mall, Boston, MA 02215}
\email{antter@bu.edu}
\address{Department of Mathematics, Texas State University, 601 University Drive, San Marcos, TX 78666}
\email{yang@txstate.edu}

\begin{document}

\maketitle

\begin{abstract}
We study the orbits under the natural action of a permutation group $G \leq \Sy_n$ on the powerset $\mathscr{P}(\{1, \dots , n\})$. The permutation groups having exactly $n+1$ orbits on the powerset can be characterized as set-transitive groups (see Definition \ref{def2.2}) and were fully classified in \cite{BP55}. In this paper, we establish a general method that allows one to classify the permutation groups with $n+r$ set-orbits for a given $r$, and apply it to integers $2 \leq r \leq 15$ with the help of the computer algebra system GAP ~\cite{GAP2020}.
\end{abstract}

\section{Introduction}
Throughout the paper, we let $n\geq 2$ denote a positive integer and let $N=\{1, \dots, n\}$.\footnote{We exclude case $n=1$ as the only group action on the set \{1\} is the trivial action.}  By a \textit{permutation group on n letters} we mean a subgroup $G$ of $\Sy_n$ endowed with the natural action $(g,x) \mapsto gx \coloneqq g(x) : G \times N \to N$. We call $n$ the \emph{degree} of the permutation group $G$. The action of $G$ on $N$ induces an action of $G$ on $\mathscr{P}(N)$, given by $(g,X)\mapsto gX =\{gx: x\in X\} : G \times \mathscr{P}(N) \to \mathscr{P}(N)$. In this case, the elements being acted on are subsets of $N$. Accordingly, we shall call the orbits under this action \textit{set-orbits}. Note that for all $g \in G$ and $X \subseteq N$, $|gX|=|X|$. Since there are $n+1$ distinct sizes of subsets of $N$, it follows that there are at least $n+1$ distinct set-orbits. Additionally, this shows that all sets in the same set-orbit will have the same cardinality. A set-orbit containing sets of size $t$ will be called a $t$-\textit{set-orbit}. For a given permutation group $G$ on $n$ letters, the number of distinct $t$-set-orbits under the action of $G$ on $\mathscr{P}(N)$ will be denoted by $s_t(G)$ and the total number of set-orbits will be denoted by $s(G)$. Clearly $s(G)=\sum_{s=0}^ns_t(G)$.

\begin{define}[Beaumont and Peterson, \cite{BP55}]
Given an integer $0\leq t\leq n$, a permutation group $G$ on $n$ letters is called \emph{t-set-transitive} if for all $t$-element subsets $S,T \subseteq N$, there exists $g\in G$ such that $g S=T$.
\end{define}

In terms of the action of $G$ on $\mathscr{P}(N)$, we see that $G$ is $t$-set-transitive if and only if $s_t(G)=1$. In other words, a $t$-set-transitive group is a permutation group with exactly one $t$-set-orbit. Clearly, all permutation groups are $0$-set-transitive and $n$-set-transitive. A $1$-set-transitive permutation group is simply transitive on $N$.

\begin{define}[Beaumont and Peterson, \cite{BP55}]\label{def2.2}
A permutation group $G$ on $n$ letters is called \emph{set-transitive} if $G$ is $t$-set-transitive for all integers $0\leq t\leq n$.
\end{define}

Set-transitive groups were studied as early as 1944 by Neumann and Morgenstern \cite{MORNEU}. In \cite{BP55}, Beaumont and Peterson proved that a set-transitive permutation group on $n$ letters, with $n \notin \{5,6,9\}$, always contains the alternating group $\Al_n$. There has also been significant research devoted to bounding the number of set-orbits $s(G)$ of a degree $n$ permutation group $G$. A trivial lower bound is $s(G) \geq 2^n/|G|$. In \cite{CAMERON}, Cameron proved that if $G$ has order $\exp(o(n^{1/2}))$, then $s(G)= (2^{n}/|G|)(1 + o(1))$. In \cite{BAPYBER}, Babai and Pyber showed that if $G$ does not contain $\Al_l$ ($l>t\geq 4$) as a composition factor, then $\frac {\log_2 s(G)} n \geq \frac {c} t$ for some positive constant $c$. In the same paper, they raised the following question: what is $\inf(\frac {\log_2 s(G)} n)$ over all solvable degree $n$ permutation groups $G$? This question was answered in a recent paper by Yang in \cite{YY11}.\\
\indent On the other hand, there has been relatively little work done on the problem of classifying groups in terms of their number of set-orbits. Beaumont and Peterson successfully classified all set-transitive permutation groups in \cite{BP55}, and Kantor classified $2,3,4$-set-transitive groups which are not $2,3,4$-transitive in \cite{KAN}. These classifications lend itself to a natural generalization in the following sense. Viewed in terms on the action of $G$ on $\mathscr{P}(N)$, we see that $G$ is set-transitive if and only if $s(G)=n+1$. Our paper seeks to completely classify the permutation groups $G$ on $n$ letters satisfying $s(G)=n+r$, for small positive integers $r$. The paper is laid out in the following manner. In Sections 1 and 2, we give necessary definitions and useful facts. In Section 3, we develop a general method mimicking the strategy in \cite{BP55}. In Section 4, we exemplify the method by classifying groups with $n+2$, $n+3$, $n+4$, and $n+5$ set-orbits. Then we use GAP ~\cite{GAP2020} to calculate all such groups for $r \leq 15$. Before continuing, we state the following useful facts from \cite{BP55} and \cite{LW65}.

Given a permutation group $G$ on $n$ letters:
\begin{enumerate}
    \item If $G$ contains a $t$-set-transitive subgroup $H$, then $G$ is $t$-set-transitive.
    \item If $G$ is $t$-transitive, then $G$ is $s$-set-transitive for all positive integers $s\leq t$.
    \item The symmetric group $\Sy_n$ is set-transitive.
    \item The alternating group $\Al_n$ is set-transitive for all $n\geq 3$.
    \item If $s_t(G) = 1$ for some $2 \leq t \leq \left \lfloor \frac{n}{2} \right \rfloor$, then $G$ is primitive.\footnote{Throughout this paper, we let $\lfloor x \rfloor$ be the floor function.}
\end{enumerate}

\section{Main Theorems}

Theorem 3 in \cite{BP55} states that if $G$ is $t$-set-transitive, then $G$ is $(n-t)$-set-transitive for any integer $1\leq t\leq n-1$. It is easy to generalize this to the following observation:

\begin{lem}\label{lem3.1}
    Let $G$ be a permutation group on $n$ letters and let $0\leq t\leq n$ be an integer. Then $s_t(G)=s_{n-t}(G)$.
\end{lem}

%\begin{proof}
%  This is an easy observation.
%\end{proof}
   %For each $\mathcal{S}=\{X_1, \dots ,X_{\beta}\} \subseteq \mathscr{P}(N)$, let $\overline{\mathcal{S}}=\{X_1^c, \dots, X_{\beta}^c\}$, where $X_i^c=N-X_i$ for each $1\leq i\leq \beta$. Let $\mathcal{O} = \left \{ T_1, \dots, T_\alpha \right \}$ be a $t$-set-orbit. We claim that $\overline{\mathcal{O}} = \left \{ T^c_1, \dots, T^c_\alpha \right \}$ is an $(n-t)$-set-orbit. It suffices to show that $gT^c_i \in \overline{\mathcal{O}}$ for all $g \in G$ and all $T^c_i \in \overline{\mathcal{O}}$. Let $g \in G$ and $T^c_i \in \overline{\mathcal{O}}$. Since $\mathcal{O}$ is a set-orbit, there exists $T_j \in \mathcal{O}$ such that $gT_i = T_j$. Since $g$ is bijective, $gT^c_i = T^c_j \in \overline{\mathcal{O}}$. Thus, the map $\mathcal{O} \xmapsto{\phi} \overline{\mathcal{O}}$ from the set of $t$-set-orbits to the set of $(n-t)$-set-orbits is well-defined. By similar reasoning, the map $\mathcal{Q} \xmapsto{\varphi} \overline{\mathcal{Q}}$ from the set of $(n-t)$-set-orbits to the set of $t$-set-orbits is well-defined. And clearly $\varphi = \phi^{-1}$. Hence, it follows that the set of $t$-set-orbits and the set of $(n-t)$-set-orbits are in bijection. So $s_t(G)=s_{n-t}(G)$.

Note that there are $n+1$ distinct sizes of sets in $\mathscr{P}(N)$. Hence, if $n$ is odd, there are an even number of distinct set sizes, and so Lemma ~\ref{lem3.1} implies that a permutation group on $n$ letters must have an even number of set-orbits. Therefore, if $n$ is odd and a permutation group $G$ on $n$ letters has $s(G)=n+r$, then $r$ will have to be odd as well. Now consider the situation where $n$ is even and a permutation group $G$ on $n$ letters has an odd number of set-orbits. In this case, since $s_t(G)=s_{n-t}(G)$ for all $0\leq t \leq n$, it follows that $s_{n/2}(G)$ must be odd. We summarize these results in the following lemma.

\begin{lem}\label{cor3.2}
Let $G$ be a permutation group on $n$ letters and suppose $s(G)=n+r$.
\begin{enumerate}
    \item If $r$ is even, then $n$ is even.
    \item If $r$ is odd and $n$ is even, then $s_{n/2}(G)$ is odd.
\end{enumerate}
\end{lem}

%Observe the following theorem from \cite{LW65}.

\begin{thm}\label{thm2.3}[Livingstone and Wagner, \cite{LW65}]
    Given a permutation group $G$ on $n$ letters and an integer $1\leq t\leq \frac{n}{2}$, we have $s_{t-1}(G) \leq s_t(G)$.
\end{thm}

\begin{remark}
This theorem implies that if $1\leq t <\left\lfloor \frac{n}{2} \right\rfloor - 1$ then $s_t(G) > 1$ only if $s_{t-1}(G) > 1$. We will make frequent use of this fact in our classification.
\end{remark}

\begin{lem}\label{lem3.4}
    Let $G$ be $\left (\left \lfloor\frac{n}{2} \right \rfloor+k \right)$-set-transitive for a positive integer $k$. If there exists a prime $p$ such that $\left \lfloor\frac{n}{2} \right \rfloor+k < p \leq n$, then $G$ is $(n-p+1)$-transitive.
\end{lem}

\begin{proof}
    By Theorem 7, Corollary 1 in \cite{BP55}, it suffices to show that $p> \max \left( \left \lfloor\frac{n}{2} \right \rfloor+k, n- \left \lfloor\frac{n}{2} \right \rfloor-k \right) = \left \lfloor\frac{n}{2} \right \rfloor+k$. This is true by assumption, so we are done.
\end{proof}

\begin{lem} \label{lem3.5}
    Let $k$ be a positive integer and let $G$ be a permutation group on $n$ letters that does not contain $\Al_n$. If there exists a prime $p$ such that $\left \lfloor\frac{n}{2} \right \rfloor+k < p < \frac{2n}{3}$, then $G$ is not $\left (\left \lfloor\frac{n}{2} \right \rfloor+k \right)$-set-transitive.
\end{lem}

\begin{proof}
    Assume for contradiction that $G$ is $\left (\left \lfloor\frac{n}{2} \right \rfloor+k \right)$-set-transitive, then by Lemma \ref{lem3.4} such a group $G$ is $\left(n-p+1 \right)$-transitive. Since $G$ does not contain $\Al_n$, $G$ is at most $(\frac{n}{3}+1)$-transitive (see \cite[page 152]{Burnside}). Now notice that $n-p+1 > n-\frac{2n}{3} + 1$ since $p<\frac{2n}{3}$ and since  $n-\frac{2n}{3} + 1= \frac{n}{3}+1$, we have reached a contradiction and thus are done.
\end{proof}

Note that both Lemmas \ref{lem3.4} and \ref{lem3.5} hold for $k=0$ when $n$ is even. If we find a maximum $k_0$ for which there exists a prime $p$ such that $\left\lfloor\frac{n}2 \right\rfloor + k_0 < p < \frac23 n$, then $G$ cannot be $\left\lfloor\frac{n}2 \right\rfloor + k$ set-transitive for any $k \leq k_0$. If $n$ is odd and such a $k_0$ exists, then $G$ is not $\left\lfloor\frac{n}2 \right\rfloor + 1$ set-transitive. Thus it is not $\left\lfloor\frac{n}2 \right\rfloor$ set-transitive either. These results are summarized in the following corollary.

\begin{cor}\label{cor3.6}
    Let $G$ be a permutation group on $n$ letters not containing $\Al_n$, where $n$ is even (odd). Let $k_0$ be the greatest nonnegative (positive) integer such that there exists a prime $p$ with $\left\lfloor \frac{n}{2} \right\rfloor + k_0 < p < \frac{2n}{3}$. Then for all $0 \leq k \leq k_0$, $G$ is not $\left\lfloor\frac{n}2 \right\rfloor + k$ set-transitive.
\end{cor}
    \begin{proof}
        If $G$ is even, then the existence of such a nonnegative $k_0$ shows that there exists a prime $p$ such that $\frac{1}{2}+k \leq \frac{n}{2} + k_0 < p < \frac{2n}{3}$. By Lemma \ref{lem3.5}, $G$ is not $\frac{n}{2} + k$ set-transitive. For the case of $G$ being odd, a positive $k_0$ shows that there exists a prime $p$ such that $\left\lfloor\frac{n}{2} \right\rfloor+k \leq \left\lfloor  \frac{n}{2} \right\rfloor + k_0 < \frac{2n}{3} $. Since this holds at least for $k = 1$, it holds for $k = 0$ since $n - \left( \left\lfloor\frac{n}{2} \right\rfloor+1\right) = \left\lfloor\frac{n}{2} \right\rfloor$.
    \end{proof}

\begin{thm}\label{thm3.7}
    A permutation group $G$ on $n$ letters not containing $\Al_n$ is not $\left \lfloor \frac{n}{2} \right \rfloor + k$ set-transitive for any positive integer values of $k \leq k_0$ where $k_0$ is the largest integer such that $48 - \frac{n+1}{2} \leq k_0 \leq \frac{5}{54}n - \frac12$.
\end{thm}
    \begin{proof}
     First note that $48 \leq \frac{n+1}2 + k_0$. Due to a result by Breusch \cite{Breusch}, which states that there exists a prime between $x$ and $\frac{9x}8$ for $x \geq 48$, there exists a prime $p$ between $\frac{n+1}2 + k_0$ and $\frac9{16}n + \frac{9}8 k_0 + \frac9{16}$. Note that $p >  \left \lfloor \frac{n}2 \right \rfloor + k$ for all $k \leq k_0$. Since $k_0 \leq \frac{5}{54}n - \frac12$, we have $p < \frac9{16}n + \frac{9}8 k_0 + \frac9{16} \leq \frac{9}{16}n + \frac{9}{8}\left( \frac5{54}n - \frac12 \right) + \frac9{16} = \frac23 n.$ Thus by Lemma \ref{lem3.5}, $G$ cannot be $\left \lfloor \frac{n}{2} \right \rfloor + k$ set-transitive for any $k \leq k_0$.
    \end{proof}

\begin{remark}
For $k = 0$ the theorem still holds true for even $n$. For odd $n$, $G$ is not $\left \lfloor\frac{n}{2}\right\rfloor$ set-transitive when a positive $k_0$ exists because $G$ will not be $\left\lfloor \frac{n}{2}\right\rfloor+1$ set-transitive.
This theorem is powerful as it shows that often sets of the same size lie in different orbits. The larger $n$ is, the more set-orbits $G$ will have. The theorem can be applied to make an upper bound on the amount of letters $G$ can permute and have exactly $n+r$ set-orbits. For example, the first value of $n$ for which we get an applicable $k_0$ is $n=81$, which gives $k_0 = 7$. This implies that a permutation group $G$ on $n$ letters not containing $\Al_n$ has at least $14$ additional set-orbits, which leads to a corollary.
\end{remark}

\begin{cor}
If a permutation group $G$ on $n$ letters that does not contain $\Al_n$ has less than $n+16$ set-orbits, then $n \leq 81$.
\end{cor}

The following Lemmas are in ~\cite{BAPYBER}.

\begin{lem}  \label{lemnew0}
If $L \leq G \leq \Sym(\Omega)$, then $s(G) \leq s(L) \leq s(G) \cdot |G:L|$.
\end{lem}

\begin{lem}  \label{lemnew1}
Assume $G$ is intransitive on $\Omega$ and has orbits $\Omega_1, \dots , \Omega_m$. Let $G_i$ be the restriction of $G$ to $\Omega_i$. Then \[s(G) \geq s(G_1) \times \cdots \times s(G_m).\]
\end{lem}
\begin{proof}
Since $G \leq G_1 \times \cdots \times G_m$ we can apply Lemma ~\ref{lemnew0}. Clearly $s(G_1 \times \cdots \times G_m)=s(G_1) \times \cdots \times s(G_m)$.
\end{proof}

\begin{lem}  \label{lemnew2}
Let $G$ be a transitive permutation group acting on a set $\Omega$ where $|\Omega|=n$. Let $(\Omega_1, \dots, \Omega_m)$ denote a system of imprimitivity of $G$ with maximal block-size $b$ $(1 \leq b <n; b = 1$ if and only if $G$ is primitive; $bm=n$). Let $N$ denote the normal subgroup of $G$ stabilizing each of the blocks $\Omega_i$. Let $G_i=\Stab_G(\Omega_i)$, and denote $s=s(G_1)$. Then

\[s(G) \geq {{s+m-1} \choose s-1}.\]
%\begin{enumerate}
%\item $s(G) \geq s^m/|G/N|$.
%\item $s(G) \geq {{s+m-1} \choose s-1}$.
%\end{enumerate}
\end{lem}

%{\color{red} These lemmas will be instrumental in solving the cases of $12 \leq r \leq 15$. We will show their exact use later in the paper.}

\section{Outline of Methods}
    In this section we will outline a step-by-step method on how we fully classify groups with $n+r$ set-orbits for $2\leq r \leq 5$ which can also be applied to classify groups with $n+r$ set-orbits for even greater $r$. To outline our method we first reduce the amount of letters $n$ on which $G$ could act, then once we have a reasonable sized list we can test a number of specific permutation groups for the remaining $n$ values. Throughout the whole method we assume that any permutation group we consider does not contain $\Al_n$ because if it did then $s(G) = n+1$. We will exemplify how to do some of the steps in the calculation sections of this paper.

    \subsection*{Step 0}
        Choose the $r$ value for which you want to classify all groups with $s(G) = n+r$. Let $k_0 =\left\lfloor \frac{r-1}{2} \right\rfloor$ and find the smallest $n$ such that $k_0$ fits in the bounds specified by Theorem \ref{thm3.7}. The smallest $n$ for which any $k_0$ appears is $n=81$ with $k_0 = 7$. This is ``Step $0$" because for any reasonable $r$, say $r<16$, we know that $n \leq 81$.
    \subsection*{Step 1}
        Now that we have an upper bound on $n$ we can start eliminating some of the possible $n$. Right away, if $r$ is even then we can eliminate all the odd $n$ by Lemma \ref{cor3.2}. Since we need $n+r$ set-orbits, we know that $G$ cannot be $s$ set-transitive for at most $r-1$ different set sizes $s$. This is where we can use Corollary \ref{cor3.6}. If $n$ is odd then we need a $k_0$ value such that there is a prime $\frac{n-1}{2}+ k_0 < p< \frac{2n}{3}$ and $2k_0> r-1$. If $n$ is even then we need a $k_0$ value such that there is a prime $\frac{n}{2} + k_0 < p < \frac{2n}{3}$ and $2k_0 + 1 > r-1$. This is because then we would know that $s_t(G) > 1$ for $2k_0$ different $s$ values in the odd case, and $2k_0+1$ different $s$ values in the even case (due to the fact that $s_\frac{n}{2}(G) = s_{n - \frac{n}{2}}(G)$). If we look at a table of primes and find any such $p$ values for the necessary $k_0$ for a given $n$, then we can remove that $n$ from the list of candidates.
    \subsection*{Step 2: Miller's Method}
        Now we look at our remaining $n$ values and apply a theorem of Miller \cite[vol. \RNum{3}, p. 439]{Miller}  which states that if $n=mp_0+r$, $p_0$ is prime, $m \in\mathbb{N}$, $p_0>m$, $r>m$, then a group $G$ on $n$ symbols, not containing $\Al_n$, cannot be more than $r$-transitive. We decompose $n$ so we have values of $m$, $p_0$, and $r$ that fit the conditions and we try to find a sufficiently small $r$. Using Lemmas \ref{lem3.4} and \ref{lem3.5}, we see if we can find a small enough $p$ such that $\left\lfloor  \frac{n}{2}\right\rfloor + k_1 < p \leq n $, then $n - p + 1 > r$ will contradict that $G$ is $\left\lfloor  \frac{n}{2}\right\rfloor + k_1$ set-transitive. For this, $k_1$ is the same as $k_0$ in the last step but it is not necessary that $\left\lfloor\frac{n}{2} \right\rfloor + k_1 < \frac{2n}{3}$. If we reach a contradiction, then we can remove that $n$ from the list.
    \subsection*{Step 3}
        After applying Miller's Method we have reduced the number of possible groups on $n$ letters that can have $n+r$ set-orbits. We will now reduce the number of groups even further using the following argument. For a given $n$ and $r$ where $r<n-4$, if $s_2(G)>1$ then it follows that $s(G) > n+r$, in which case $G$ does not have $n+r$ set-orbits. Thus, we can assume $s_2(G) = 1$ which implies that $G$ is primitive by Theorem 6 in \cite{BP55}. Similarly, if $r<n-2$, $s_1(G)>1$ implies that $s(G) > n+r$, in which case $G$ does not have $n+r$ set-orbits. By assuming $s_1(G)=1$ in this case, we know that $G$ is transitive by Theorem 5 in \cite{BP55}. In these two cases, we will use \cite{TGps} to find the structure of the transitive groups. To reduce even further, we will introduce a fact from \cite{BP55} which follows simply from the orbit stabilizer theorem.
            \begin{fact}
                If a permutation group $G$ on $n$ letters is $s$ set-transitive, then $n \choose s$ divides $|G|$.
            \end{fact}

        \noindent Since we know $s_{k_1}(G) = 1$, then $n \choose {k_1}$ must divide $|G|$.  Under the above assumption, $G$ is primitive or transitive, in which case one only needs to check a few groups. If the above restriction on $r$ and $n$ does not hold, then we cannot use the fact that $s_t(G) = 1$ for any $t$. Under this situation, one must check all non-trivial subgroups of $\Sy_n$. We can use GAP to compute these cases.

    \subsection*{Step 4: Computation} At this point, we have a list of possible $n$ for the degree of $G$, and a list of possible groups for each $n$. Now we can simply run a GAP program to calculate $s(G)$ for each candidate, and list out the ones that have $s(G) = n+r$ as desired.

    \begin{remark}
        We will do this process by hand for $r = 2,3,4,5$ to exemplify the process and then use a GAP program to go up to $r = 15$. Since we are going in a linear order for $r = 2,3,4,5$, we will often run into the same group twice. An example of this is when we consider all primitive groups on eight letters such that ${8 \choose 3} = 56$ divides $|G|$, and in a later section we consider all primitive groups such that ${8 \choose 2} = 28$ divides $|G|$. We will not consider the same groups twice if we have already calculated $s(G)$, but rather just list the groups we know have $s(G) = n+r$ from previous sections and then only consider the groups whose order is divisible by $28$ but not $56$. There will be several different instances where we can use previous knowledge to reduce the possible number of groups with $n+r$ set-orbits.
    \end{remark}

\section{Groups With Few Set-Orbits}

\noindent
\textbf{Groups with $n+2$ set-orbits} Assume a group $G$ on the set $N = \lbrace 1, \ldots, n\rbrace$ has $n+2$ set-orbits.

Looking at only even $n$ and applying Step 1 with $k_0=1$, we are left with the following  possibilities: $$n = 2,4,6,8,10,12, 14,16, 24.$$
To exemplify Step 2, we will show one of the calculations done. Since $24 = 1 \times 19 + 5$, then $G$ is at most $5$-transitive. Using $k_1 = 1$  we need to find the smallest prime $p$ such that $13 < p \leq  24$, so $p = 17$. Thus, if $G$ was $13$ set-transitive, then it would be $n-p+1 = 8$-transitive, which would contradict that it is at most $5$-transitive. Thus, we know a permutation group on $24$ letters cannot have $s(G) = n+2$.

We spare the reader from having to see any more of these calculations. At the end of this method, we are left with $$n = 2, 4, 6, 8, 12.$$

Before continuing to Step 3, we take care of the trivial case of $n=2$. The only permutation groups on $2$ letters are the trivial group and $\Sy_2$. It happens that the trivial group has $s(G) = 4$, so we must include it. We will no longer consider $n=2$ for any $r$.

Now we move on to Step 3. We consider transitive groups on $4$ letters and primitive groups on $6,8,12$ letters such that $15,56,792$ divides the group orders, respectively.

Now we move on to Step 4 (computation). We show all of the groups for which $s(G) = n+2$ in the following table. For the GAP ID we let nTr be TransitiveGroup(n,r), nPr be PrimitiveGroup(n,r).  We let nSr be  ConjugacyClassesSubgroups($S_n$)$[\mathrm{r}]$.

\begin{center}
\begin{tabular}{|c|c|c|c|}
   \hline
    $n$ & $G$ & $|G|$ & GAP ID \\ \hline
     $2$& $1$ & $1$ & 2S1 \\ \hline
     $4$ & $C_4$ &$4$ & 4T1 \\ \hline
     $4$ & $D_8$ & $8$ & 4T3 \\ \hline
     6 & $\rm{PSL}$$(2,5)$ & $60$ & 6P1 \\ \hline
     $8$ & $\rm{AGL}$$(1,8)$ & $56$ & 8P1 \\ \hline
     $8$ & $\rm{A\Gamma L}$$(1,8)$ & $168$ & 8P2 \\ \hline
     $8$ & $\rm{PGL}$$(2,7)$ & $336$ & 8P5 \\ \hline
     $8$ & $\rm{ASL}$$(3,2)$ & $1344$ & 8P3 \\ \hline
   $12$ & $M_{12}$ & $95040$ & 12P2 \\ \hline
\end{tabular}
\end{center}

\noindent
\textbf{Groups with $n+3$ set-orbits} From the previous computations we know that $C_2 \times C_2$ on $4$ letters and $\rm{PSL}(2,7)$ on eight letters have $s(G) = n+3$.

For Step 1, we use a $k_0$ value of 2 for odd $n$ and 1 for even $n$. We find primes in the necessary range and are left with $$n = 3,4, \dots, 16, 19, 23, 24, 25, 43.$$

Applying Step 2, we are left with $$n = 3,4,5,6,7,8,9,11,12.$$

For $n=3$, we check all subgroups of $\Sy_3$. For $n=4,5$ we check the transitive groups. For $n=6,7,8,9,11,12$ we check primitive groups with order divisible by $15,21,56, 84,330, 792$, respectively. Note that for the even $n$ we are already done from the previous section. We show the results for all groups with $n+3$ set-orbits in a table, using the same GAP identification key as in the previous section.

\begin{center}

\begin{tabular}{|c|c|c|c|}
   \hline
    $n$ & $G$ & $|G|$ & GAP ID \\ \hline
     $3$& $C_2$ & $2$ & 3S2 \\ \hline
    $4$ & $C_2 \times C_2$ & $4$ & 4T2 \\ \hline
     $5$ & $C_5$ & $5$ & 5T1 \\ \hline
     5 & $D_{10}$ & $10$ & 5T2 \\ \hline
     $7$ & $\rm{AGL}$$(1,7)$ & $42$ & 7P4 \\ \hline
     $7$ & $\rm{PSL}$$(3,2)$ & $168$ & 7P5 \\ \hline
     $8$ & $\rm{PSL}$$(2,7)$ & 168 & 8P4 \\ \hline
      $11$& $M_{11}$ & 7920 & 11P6 \\ \hline
\end{tabular}
\end{center}

\noindent
\textbf{Groups with $n+4$ set-orbits} We know that the same numbers we were unable to remove in the earlier sections will reappear. Since for even numbers we need $k_0 = 2$, the following $n$'s can no longer be removed using Step 1: $n=18,22,34,42$. So we must continue with $$n = 4, 6,8,10,12,14,16,18,22,24,34,42.$$

After Step 2 we have only $$n = 4,6,8,10,12.$$

So we must check all subgroups of $\Sy_4$, the transitive groups on $6$ letters, and the primitive groups on $8,10,12$ letters that whose order is divisible by $28, 120, 495$. We do not recheck the primitive groups on $8$ letters divisible by 56.

The table of all groups with $s(G) = n+4$ is listed below.

\begin{center}

\begin{tabular}{|c|c|c|c|}
   \hline
    $n$ & $G$ & $|G|$ & GAP ID \\ \hline
     4& $C_3$ & $3$ & 4S4 \\ \hline
     4 & $\Sy_3$ & $6$ & 4T8 \\ \hline
     $6$ & $C_3 \times \Sy_3$ & 18 & 6T5 \\ \hline
     $6$ & $\Sy_4$ & $24$ & 6T8 \\ \hline
     $6$ & $\Sy_3 \times \Sy_3$ & 36 & 6T9 \\ \hline
     $6$& $(C_3 \times C_3) \rtimes C_4$ & 36 & 6T10 \\ \hline
     $6$ & $C_2 \times \Sy_4$ & 48 & 6T11 \\ \hline
     6 & $(C_3 \times C_3) \rtimes D_8$ & 72 & 6T13 \\ \hline
     $10$ & $\rm{PGL}$$(2,9)$ & $720$ & 10P4 \\ \hline
     $10$ & $\rm{P\Gamma L}$$(2,9)$ & $1440$ & 10P7 \\ \hline

\end{tabular}
\end{center}

\noindent
\textbf{Groups with $n+5$ set-orbits} From previous sections we have found five groups with $n+5$ set-orbits. For even $n$ we continue to use $k_0=2$, and for odd $n$ we take $k_0=3$. After Step 1 $\&$ 2 we have the same even numbers as in the above section, so $n=4,6,8,10,12$. The odd numbers after Step 1 will be the same as in the $n+3$ case, but now we must add $n=17,21,33,41$. We must apply Step 2 to the odd numbers $$n=5, 7, 9, 11, 13, 15, 17, 19, 21, 23, 25, 33, 41, 43.$$

After applying Step 2 we have $$n= 4,5,6,7,8,9,10,11,12.$$

Note that once again, we handled all the necessary computations for the even $n$ in the previous section. So we must check all subgroups of $\Sy_5$, the transitive subgroups of $\Sy_7$, and the primitive groups on $9,11$ letters whose order is divisible by $36,165$ but not $84, 330$, respectively. We list all groups with $n+5$ set-orbits in the table below.

\begin{center}

\begin{tabular}{|c|c|c|c|}
   \hline

 $n$ & $G$ & $|G|$ & GAP ID \\ \hline
     3& $1$ & $1$ & 3S1 \\ \hline
     4 & $C_2 \times C_2$ & 4 & 4S6 \\ \hline
     5 & $\Al_4$ & $12$ & 5S14 \\ \hline
     5 & $\Sy_4$ & $24$ & 5S17 \\ \hline
     $6$ & $C_2 \times \Al_4$ & 24 & 6T6 \\ \hline
     $6$ & $\Sy_4$ & $24$ & 6T7 \\ \hline
     $7$ & $C_7 \rtimes C_3$ & 21 & 7P3 \\ \hline
     $9$& $\rm{ASL}$$(2,3)$ & 216 & 9P6 \\ \hline
     $9$ & $\rm{AGL}$$(2,3)$ & 432 & 9P7 \\ \hline
     $10$ & $M_{10}$ & 720 & 10P6 \\ \hline
\end{tabular}
\end{center}

\noindent
\textbf{Remaining Computations}
We indeed classify all the cases till $r=15$ using the same method.

%{\color{red}

The cases of $12 \leq r \leq 15$ require some additional work. In each of these cases we could check all subgroups of $\Sy_n (n\leq 11)$ using GAP. Since our computers cannot use GAP to compute all subgroups of $\Sy_n (n \geq 12)$,  we will use Lemmas  \ref{lemnew1} and \ref{lemnew2} to eliminate non-transitive or imprimitive subgroups of $\Sy_n$. It's important to note that for transitive and primitive subgroups of $\Sy_n$ we can still use GAP since GAP's built in library has all primitive  groups of degree less than or equal to $4096$. We shall discuss the case when $r=12$ in detail and then list the results for $13 \leq r \leq 15$ as the process will be very similar. %}

%{\color{red}
In the case $r=12$, note that we may check the number of set-orbits of all the subgroups of $\Sy_n$ ($n\leq 11$) using GAP, and we can also check all the transitive and primitive subgroups of $\Sy_n$ for $n$ not too large. We will discuss how to handle the case when $n=12$, and the other possible $n$ values can be checked in a similar way.

Assume the action of group $G$ is not transitive, then the $1$-set-orbits will be partitioned into at least 2 orbits. Assume it is partitioned into at least 3 orbits, then $s_1 \geq 3$ and by Theorem \ref{thm2.3}, we have $s_i \geq 3$ for $1\leq i \leq 11$. Thus $s(G)\geq 3\cdot 11+2 = 35 = 12+23$, and this is impossible. Then assume the $1$-set-orbits is partitioned in two, then $G$ can be a subgroup of $\Sy_1\times \Sy_{11}, \Sy_2\times \Sy_{10}, \Sy_3\times \Sy_9, \Sy_4\times \Sy_8, \Sy_5\times \Sy_7, \Sy_6\times \Sy_6$. By Lemma \ref{lemnew1} we check that in the corresponding cases, the set-orbits will be at least $24, 33, 40, 45, 48, 49$. In which case the only possibility is a subgroup of $\Sy_1\times \Sy_{11}$, and in this case, $G$ is either $\Sy_{11}$ or $\Al_{11}$.

%When $n = 13$, the smallest possible number of set orbit appears when a group is a subgroup of $\Sy_1\times \Sy_{12}$, in which case the number of set orbits is at least $2\cdot 13 = 26 = 13+13$. Therefore for $n>13$, intransitive cases are impossible.

Assume the action of group $G$ is transitive but not primitive, note that $12=2 \cdot 6$, $3\cdot 4$,  $4\cdot 3$, or $6\cdot 2$, and by applying Lemma  ~\ref{lemnew2}, we see that all the cases will lead to more set-orbits.

Assume the action of group $G$ is transitive and  primitive, we use GAP to run through all the possible primitive groups of degree $12$ and we see none of them satisfy the requirement.\\ %}

We list all the results in the next few tables.

Groups with $n+6$ set-orbits:

\begin{center}
\begin{longtable}{|c|c|c|c|}
   \hline
   $n$ & $G$ & $|G|$ & GAP ID \\ \hline
    4 &   $C_{2} $ & 2 & 4S2 \\ \hline
6 &   $\Al_{4} $ & 12 & 6S31 \\ \hline
6 &   $C_{5} \rtimes C_{4} $ & 20 & 6S38 \\ \hline
6 &   $\Al_{5} $ & 60 & 6S50 \\ \hline
6 &   $\Sy_{5} $ & 120 & 6S53 \\ \hline
\end{longtable}
\end{center}

Groups with $n+7$ set-orbits:

\begin{center}
\begin{longtable}{|c|c|c|c|}
   \hline
   $n$ & $G$ & $|G|$ & GAP ID \\ \hline
    5 &   $C_{4} $ & 4 & 5S6 \\ \hline
5 &   $\Sy_{3} $ & 6 & 5S10 \\ \hline
5 &   $C_{6} $ & 6 & 5S11 \\ \hline
5 &   $D_{8} $ & 8 & 5S12 \\ \hline
5 &   $D_{12} $ & 12 & 5S15 \\ \hline
6 &   $D_{12} $ & 12 & 6S33 \\ \hline
7 &   $\Sy_{5} $ & 120 & 7S89 \\ \hline
7 &   $\Al_{6} $ & 360 & 7S93 \\ \hline
7 &   $\Sy_{6} $ & 720 & 7S94 \\ \hline
8 &   $\left( \left( \left( C_{2} \times C_{2} \times C_{2} \times C_{2} \right) \rtimes C_{2} \right) \rtimes C_{2} \right) \rtimes C_{3} $ & 192 & 8T38 \\ \hline
8 &   $\left( \left( \left( C_{2} \times C_{2} \times C_{2} \right) \rtimes \left( C_{2} \times C_{2} \right) \right) \rtimes C_{3} \right) \rtimes C_{2} $ & 192 & 8T40 \\ \hline
8 &   $\left( \left( \left( C_{2} \times C_{2} \times C_{2} \times C_{2} \right) \rtimes C_{3} \right) \rtimes C_{2} \right) \rtimes C_{3} $ & 288 & 8T42 \\ \hline
8 &   $\left( \left( \left( \left( C_{2} \times C_{2} \times C_{2} \times C_{2} \right) \rtimes C_{2} \right) \rtimes C_{2} \right) \rtimes C_{3} \right) \rtimes C_{2} $ & 384 & 8T44 \\ \hline
8 &   $\left( \left( \Al_{4} \times \Al_{4} \right) \rtimes C_{2} \right) \rtimes C_{2} $ & 576 & 8T45 \\ \hline
8 &   $\left( \Al_{4} \times \Al_{4} \right) \rtimes C_{4} $ & 576 & 8T46 \\ \hline
8 &   $\left( \Sy_{4} \times \Sy_{4} \right) \rtimes C_{2} $ & 1152 & 8T47 \\ \hline
9 &   $\left( C_{3} \times C_{3} \right) \rtimes C_{8} $ & 72 & 9T15 \\ \hline
9 &   $\left( C_{3} \times C_{3} \right) \rtimes QD_{16} $ & 144 & 9T19 \\ \hline
12 &   $M_{11} $ & 7920 & 12P1 \\ \hline
\end{longtable}
\end{center}

\newpage

Groups with $n+8$ set-orbits:

\begin{center}
\begin{longtable}{|c|c|c|c|}
   \hline
   $n$ & $G$ & $|G|$ & GAP ID \\ \hline
   4 &   $C_{2} $ & 2 & 4S3 \\ \hline
6 &   $C_{6} $ & 6 & 6S17 \\ \hline
8 &   $\left( \left( C_{2} \times C_{2} \times C_{2} \right) \rtimes \left( C_{2} \times C_{2} \right) \right) \rtimes C_{3} $ & 96 & 8S242 \\ \hline
8 &   $\left( \left( C_{2} \times C_{2} \times C_{2} \times C_{2} \right) \rtimes C_{2} \right) \rtimes C_{3} $ & 96 & 8S247 \\ \hline
8 &   $\left( \left( \left( C_{2} \times C_{2} \times C_{2} \right) \rtimes \left( C_{2} \times C_{2} \right) \right) \rtimes C_{3} \right) \rtimes C_{2} $ & 192 & 8S268 \\ \hline
8 &   $\left( \left( \left( C_{2} \times C_{2} \times C_{2} \times C_{2} \right) \rtimes C_{2} \right) \rtimes C_{3} \right) \rtimes C_{2} $ & 192 & 8S270 \\ \hline
8 &   $\Al_{7} $ & 2520 & 8S293 \\ \hline
8 &   $\Sy_{7} $ & 5040 & 8S294 \\ \hline
12 &   $\rm{PSL}$$(2,11) \rtimes C_{2} $ & 1320 & 12T218 \\ \hline
\end{longtable}
\end{center}

Groups with $n+9$ set-orbits:

\begin{center}
\begin{longtable}{|c|c|c|c|}
   \hline
   $n$ & $G$ & $|G|$ & GAP ID \\ \hline
   5 &   $C_{2} \times C_{2} $ & 4 & 5S5 \\ \hline
6 &   $C_{2} \times \Al_{4} $ & 24 & 6S40 \\ \hline
6 &   $\Sy_{4} $ & 24 & 6S41 \\ \hline
6 &   $C_{2} \times \Sy_{4} $ & 48 & 6S49 \\ \hline
7 &   $\Al_{5} $ & 60 & 7S81 \\ \hline
8 &   $\left( \left( C_{2} \times C_{2} \times C_{2} \times C_{2} \right) \rtimes C_{3} \right) \rtimes C_{2} $ & 96 & 8S240 \\ \hline
9 &   $\left( C_{3} \times C_{3} \right) \rtimes Q_{8} $ & 72 & 9S370 \\ \hline
9 &   $\Al_{8} $ & 20160 & 9S551 \\ \hline
9 &   $\Sy_{8} $ & 40320 & 9S552 \\ \hline
10 &   $\Sy_{6} $ & 720 & 10T32 \\ \hline

\end{longtable}
\end{center}

Groups with $n+10$ set-orbits:

\begin{center}
\begin{longtable}{|c|c|c|c|}
   \hline
   $n$ & $G$ & $|G|$ & GAP ID \\ \hline
6 &   $C_{5} $ & 5 & 6S14 \\ \hline
6 &   $\Sy_{3} $ & 6 & 6S16 \\ \hline
6 &   $C_{3} \times C_{3} $ & 9 & 6S28 \\ \hline
6 &   $D_{10} $ & 10 & 6S29 \\ \hline
6 &   $\left( C_{3} \times C_{3} \right) \rtimes C_{2} $ & 18 & 6S35 \\ \hline
6 &   $C_{3} \times \Sy_{3} $ & 18 & 6S37 \\ \hline
6 &   $\Sy_{3} \times \Sy_{3} $ & 36 & 6S45 \\ \hline
8 &   $\rm{GL}$$(2,3) $ & 48 & 8S216 \\ \hline
10 &   $\Al_{6} $ & 360 & 10S1396 \\ \hline
10 &   $\rm{PSL}$$\left( 2,8\right) $ & 504 & 10S1448 \\ \hline
10 &   $\rm{PSL}$$\left(2,8 \right) \rtimes C_{3} $ & 1512 & 10S1539 \\ \hline
10 &   $\Al_{9} $ & 181440 & 10S1590 \\ \hline
10 &   $\Sy_{9} $ & 362880 & 10S1591 \\ \hline
12 &   $\rm{PSL}$$(2,11) $ & 660 & 12T179 \\ \hline
\end{longtable}
\end{center}

\newpage
Groups with $n+11$ set-orbits:

\begin{center}
\begin{longtable}{|c|c|c|c|}
   \hline
   $n$ & $G$ & $|G|$ & GAP ID \\ \hline
   5 &   $C_{3} $ & 3 & 5S4 \\ \hline
5 &   $\Sy_{3} $ & 6 & 5S9 \\ \hline
7 &   $D_{14} $ & 14 & 7S48 \\ \hline
7 &   $C_{2} \times \left( C_{5} \rtimes C_{4} \right) $ & 40 & 7S75 \\ \hline
7 &   $\Sy_{5} $ & 120 & 7S87 \\ \hline
7 &   $C_{2} \times \Al_{5} $ & 120 & 7S88 \\ \hline
7 &   $C_{2} \times \Sy_{5} $ & 240 & 7S92 \\ \hline
8 &   $\rm{SL}$$(2,3) $ & 24 & 8S154 \\ \hline
9 &   $\left( C_{2} \times C_{2} \times C_{2} \right) \rtimes C_{7} $ & 56 & 9S355 \\ \hline
9 &   $\left( \left( C_{3} \times C_{3} \times C_{3} \right) \rtimes C_{3} \right) \rtimes C_{2} $ & 162 & 9S457 \\ \hline
9 &   $\left( \left( C_{3} \times C_{3} \times C_{3} \right) \rtimes C_{3} \right) \rtimes C_{2} $ & 162 & 9S458 \\ \hline
9 &   $\left( C_{2} \times C_{2} \times C_{2} \right) \rtimes \left( C_{7} \rtimes C_{3} \right) $ & 168 & 9S462 \\ \hline
9 &   $\left( \left( C_{3} \times C_{3} \times C_{3} \right) \rtimes C_{3} \right) \rtimes \left( C_{2} \times C_{2} \right) $ & 324 & 9S497 \\ \hline

9 & $\rm{PSL}$$(3,2) \rtimes C_2$ & 336 & 9S499 \\ \hline

9 &   $\left( \left( \left( C_{3} \times C_{3} \times C_{3} \right) \rtimes \left( C_{2} \times C_{2} \right) \right) \rtimes C_{3} \right) \rtimes C_{2} $ & 648 & 9S522 \\ \hline
9 &   $\left( \left( \left( C_{3} \times C_{3} \times C_{3} \right) \rtimes \left( C_{2} \times C_{2} \right) \right) \rtimes C_{3} \right) \rtimes C_{2} $ & 648 & 9S524 \\ \hline
9 &   $\left( \left( \left( \left( C_{3} \times C_{3} \times C_{3} \right) \rtimes \left( C_{2} \times C_{2} \right) \right) \rtimes C_{3} \right) \rtimes C_{2} \right) \rtimes C_{2} $ & 1296 & 9S534 \\ \hline
9 &   $\left( C_{2} \times C_{2} \times C_{2} \right) \rtimes \rm{PSL}(3,2) $ & 1344 & 9S535 \\ \hline
10 &   $\left( C_{5} \times C_{5} \right) \rtimes \left( \left( C_{4} \times C_{4} \right) \rtimes C_{2} \right) $ & 800 & 10S1496 \\ \hline
10 &   $\left( C_{2} \times C_{2} \times C_{2} \times C_{2} \right) \rtimes \Sy_{5} $ & 1920 & 10S1542 \\ \hline
10 &   $C_{2} \times \left( \left( C_{2} \times C_{2} \times C_{2} \times C_{2} \right) \rtimes \Al_{5} \right) $ & 1920 & 10S1543 \\ \hline
10 &   $C_{2} \times \left( \left( C_{2} \times C_{2} \times C_{2} \times C_{2} \right) \rtimes \Sy_{5} \right) $ & 3840 & 10S1561 \\ \hline
10 &   $\left( \Al_{5} \times \Al_{5} \right) \rtimes C_{2} $ & 7200 & 10S1569 \\ \hline
10 &   $\left( \Al_{5} \times \Al_{5} \right) \rtimes \left( C_{2} \times C_{2} \right) $ & 14400 & 10S1576 \\ \hline
10 &   $\left( \Al_{5} \times \Al_{5} \right) \rtimes C_{4} $ & 14400 & 10S1577 \\ \hline
10 &   $\left( \Al_{5} \times \Al_{5} \right) \rtimes D_{8} $ & 28800 & 10S1584 \\ \hline
11 &   $\Al_{10} $ & 1814400 & 11S3091 \\ \hline
11 &   $\Sy_{10} $ & 3628800 & 11S3092 \\ \hline

\end{longtable}
\end{center}

Groups with $n+12$ set-orbits:

\begin{center}
\begin{longtable}{|c|c|c|c|}
   \hline
   $n$ & $G$ & $|G|$ & GAP ID \\ \hline
   4&   1  & 1 &  4S1\\ \hline
    6&  $C_4 \times C_2$   & 8 &  6S24\\ \hline
    6&  $D_8$   & 8 & 6S27 \\ \hline
    6&  $D_8 \times C_2$   & 16 & 6S34 \\ \hline
    8&  $C_7 \rtimes C_6$   & 42 & 8S196 \\ \hline
    8&  $\rm{PSL}(3,2)$   & 168 &8S264  \\ \hline
    10&  $(C_5 \times C_5 )\rtimes C_8$   &200  & 10S1311  \\ \hline
    10& $(C_5 \times  C_5) \rtimes (C_8 \rtimes C_2)$&400 &10S1418 \\ \hline
    10& $(C_2 \times C_2 \times C_2 \times C_2) \rtimes \Al_5$ &960 &10S1504 \\ \hline
    10& $(C_2 \times C_2 \times C_2 \times C_2) \rtimes \Sy_5$& 1920 & 10S1541\\ \hline
    12& $\Al_{11}$& 11!/2 &  \\\hline
    12&$\Sy_{11}$&  11! & \\ \hline

\end{longtable}
\end{center}

\bigskip

Groups with $n+13$ set-orbits:

\begin{center}
\begin{longtable}{|c|c|c|c|}
   \hline
   $n$ & $G$ & $|G|$ & GAP ID \\ \hline
5&      $C_2 \times C_2$  & 4 &  5S7\\ \hline
    6&  $D_8$   & 8 &  6S26\\ \hline
    7&  $C_7$   & 7 & 7S23 \\ \hline
    7&  $C_3 \times \Sy_3$   & 18 & 7S51 \\ \hline
    7&  $C_5\rtimes C_4$   & 20 &7S53  \\ \hline
    7&  $\Sy_4$   &24  & 7S63  \\ \hline
    7& $C_3 \times  \Al_4$&36 &7S71 \\ \hline
    7& $(C_3 \times C_3) \rtimes C_4$ &36 &7S73 \\ \hline
    7& $\Sy_3 \times \Sy_3$& 36 & 7S74\\ \hline
    7& $C_2\times \Sy_4$& 48 &7S78 \\ \hline
    7&$\Al_4\times\Sy_{3}$&  72 &7S82 \\ \hline
    7&$(C_3 \times \Al_4) \rtimes C_2$&72 & 7S83 \\ \hline
    7&$C_3 \times \Sy_4$&72 & 7S84 \\ \hline
    7&$(\Sy_3 \times \Sy_3) \rtimes C_2$&72 & 7S85 \\ \hline
    7&$\Sy_4\times \Sy_3$&144 & 7S90 \\ \hline
    8&$(C_4 \times C_4)\rtimes C_2$&32 & 8S181 \\ \hline
    8&$((C_4 \times C_4)\rtimes C_2)\rtimes C_2$&64 & 8S226 \\ \hline
    8&$((C_4 \times C_4 \times C_4)\rtimes C_4)\rtimes C_2$&64 & 8S227 \\ \hline
    8&$((C_4 \times C_4 \times C_4)\rtimes C_4)\rtimes C_2$&64 & 8S228 \\ \hline
    8&$(D_8 \times D_8)\rtimes C_2$&128 & 8S259 \\ \hline
    8&$C_2\times \Sy_5$&240 & 8S272 \\ \hline
    8&$C_2\times \Al_6$&720 & 8S286 \\ \hline
    8&$\Sy_6$&720 & 8S287 \\ \hline
    8&$C_2\times \Sy_6$&1440 & 8S292 \\ \hline
    9&$C_9\rtimes C_6$&54 & 9S354 \\ \hline
    9&  $\rm{PSL}(3,2)$   & 168 &9S460  \\ \hline
    11&  $\rm{PSL}(2,11)$   & 660 &11S2754  \\ \hline
    13& $\Al_{12}$& 12!/2 & \\ \hline
    13&$\Sy_{12}$&  12! & \\ \hline
\end{longtable}
\end{center}

Groups with $n+14$ set-orbits:

\begin{center}
\begin{longtable}{|c|c|c|c|}
   \hline
   $n$ & $G$ & $|G|$ & GAP ID \\ \hline
    6& $\Sy_3$ & 6 & 6S19 \\ \hline
   6 & $\Al_4$& 12 & 6S30\\ \hline
   6 &$\Sy_4$ & 24& 6S44\\ \hline
   8 & $C_2 \times \Al_4$& 24  & 8S157\\ \hline
    8& $(C_2 \times C_2 \times C_2) \rtimes C_4$& 32 & 8S176 \\ \hline
    8 &$C_2 \times \Sy_4$ & 48 & 8S214 \\ \hline
    8 & $((C_2 \times C_2 \times C_2 \times C_2) \rtimes C_2) \rtimes C_2$ & 64 &8S229 \\ \hline
   8 &$\Sy_5$ & 120 & 8S257\\ \hline
   10 &$(C_5 \times C_5) \rtimes D_8$ & 200 &10S1305 \\ \hline
   10 & $(C_5 \times C_5) \rtimes Q_8$ & 200   &10S1307 \\ \hline
    10& $(C_5 \times C_5) \rtimes (C_4 \times C_2)$ & 200  &10S1309 \\ \hline
   10 & $((C_2 \times C_2 \times  C_2 \times  C_2) \rtimes  C_5) \rtimes  C_4$& 320  &10S1386 \\ \hline
   10 &$(C_5 \times C_5) \rtimes ((C_4 \times C_2) \rtimes C_2)$ & 400 &10S1419 \\ \hline
   10 & $C_2 \times(((C_2 \times C_2 \times C_2 \times C_2) \rtimes C_5) \rtimes C_4) $&640  &10S1472 \\ \hline
  14 &$\Al_{13}$ & $13!/2$ & \\ \hline
   14& $\Sy_{13}$& $13!$& \\ \hline

\end{longtable}
\end{center}

Groups with $n+15$ set-orbits

\begin{center}
\begin{longtable}{|c|c|c|c|}
   \hline
   $n$ & $G$ & $|G|$ & GAP ID \\ \hline
   5 &$C_2$ & 2 & 5S3 \\ \hline
   6 & $C_2 \times C_2 \times C_2$ & 8 & 6S22 \\ \hline
   6 & $D_8$ & 8  & 6S23 \\ \hline
   7 & $C_2 \times \Al_4$ & 24  & 7S60 \\ \hline
   7 & $\Sy_4$ & 24 & 7S67 \\ \hline
   7 & $\Sy_4$ & 24 & 7S70 \\ \hline
   8 & $\Sy_4$ & 24  & 8S158 \\ \hline
   9 & $(C_3 \times C_3) \rtimes C_6$ & 54 & 9S352 \\ \hline
    9 &$(C_3 \times C_3 \times C_3) \rtimes C_3$ & 81 & 9S401 \\ \hline
   9 & $((C_3 \times C_3) \rtimes C_3) \rtimes (C_2 \times C_2)$ &108  & 9S425 \\ \hline
   9 &$ ((C_3 \times C_3 \times C_3) \rtimes C_3) \rtimes C_2$ & 162 & 9S459  \\ \hline
    9 & $((C_3 \times C_3 \times C_3)\rtimes (C_2 \times C_2) \rtimes C_3 $ & 324 & 9S496 \\ \hline
    9 & $(\Sy_3 \times \Sy_3 \times \Sy_3) \rtimes C_3$ & 648 & 9S523  \\ \hline
    9&$C_2 \times \Al_7$ & 5040 & 9S548 \\ \hline
    9& $\Sy_7$ &5040  &9S549  \\ \hline
    9& $C_2 \times \Sy_7$ & 10080 & 9S550 \\ \hline
    10& $((C_2 \times C_2 \times C_2 \times C_2) \rtimes C_5) \rtimes C_4$ & 320 & 10S1385 \\ \hline
   13 & $M_{12}$&  95040 & $\star$ \\ \hline
   15 & $\Al_{14}$& $14!/2$ & \\ \hline
   15 & $\Sy_{14}$& $14!$ & \\ \hline
\end{longtable}
\end{center}

$\star$ Here $M_{12}$ acts transitively on $12$ elements (with $14$ set-orbits in this action) and trivially on the other element.

 \section{Closing Remarks}

Now that a general method is developed for calculating all groups with $n+r$ orbits where $r$ is not too large. The GAP code used for the calculation is available at ~\cite{GAPcode}. We have successfully classified all the cases for $r \leq 15$. We remark that by far the most computationally taxing step is finding all the subgroups of $\Sy_n$. If one could come up with a method to circumvent this, the classification could go much farther. Another computationally taxing step is calculating how many set-orbits a large group has. If both of these steps can be improved upon, the classification could go further.

\section*{Acknowledgements}

This research was conducted under NSF-REU grant DMS-1757233 by Betz, Chao-Haft, Gong, and Ter-Saakov during the Summer of 2019 under the supervision of Keller and Yang. The authors gratefully acknowledge the financial support of NSF and also thank Texas State University for providing a great working environment and support. Keller and Yang were also partially supported by grants from the Simons Foundation (\#280770 to Thomas M. Keller, \#499532 to Yong Yang). The authors are grateful to the referee for the valuable suggestions which greatly improved the manuscript.

\end{document}